\documentclass[11pt]{article}
\usepackage{amssymb,amsmath,amsthm,amsfonts,latexsym}
\usepackage{tikz,subfig}
\usepgflibrary{arrows}
\newtheorem{theorem}{Theorem}[section]
\newtheorem{lemma}[theorem]{Lemma}

\usepackage[latin1]{inputenc}
\usepackage{caption}
\usepackage{afterpage}
\usepackage{fullpage}
\usepackage{float}
\restylefloat{table}

\usepackage[mathscr]{eucal}
\usepackage{pdfpages}
\usepackage[T1]{fontenc}
\usepackage[english]{babel}
\usepackage{xspace}
\usepackage{amscd,bm}
\usepackage{color}
\usepackage{graphicx}
\usepackage{verbatim}
\usepackage{multirow}
\numberwithin{equation}{section}

\newtheorem{ex}{Example}[section]

\begin{document}
\pagenumbering{arabic}
\baselineskip=1.2pc
\begin{center}
{{\bf \large Stability analysis and error estimates of local discontinuous Galerkin method for convection-diffusion equations on overlapping mesh with non-periodic boundary conditions}\footnote{Supported by NSF grant DMS-1818467 and KMITL Research Fund, Research Seed Grant for New Lecturer.}}
\end{center}
\centerline{Nattaporn Chuenjarern\footnote{Department of Mathematics, Faculty of Science, King Mongkut's Institute of Technology Ladkrabang, Bangkok, Thailand 10520. E-mail: nattaporn.ch@kmitl.ac.th}\quad Kanognudge Wuttanachamsri\footnote{Department of Mathematics, Faculty of Science, King Mongkut's Institute of Technology Ladkrabang, Bangkok, Thailand 10520. E-mail: kanognudge.wu@kmitl.ac.th} \quad
Yang Yang\footnote{Department of Mathematical Sciences,
Michigan Technological University, Houghton, MI 49931. E-mail:
yyang7@mtu.edu}}
%t\vspace{.1in}
\bigskip
\centerline{\bf Abstract}
A new local discontinuous Galerkin (LDG) method for convection-diffusion equations on overlapping meshes with periodic boundary conditions was introduced in \cite{Overlap1}. With the new method, the primary variable $u$ and the auxiliary variable $p=u_x$ are solved on different meshes. In this paper, we will extend the idea to convection-diffusion equations with non-periodic boundary conditions, i.e. Neumann and Dirichlet boundary conditions. The main difference is to adjust the boundary cells. Moreover, we study the stability and suboptimal error estimates. Finally, numerical experiments are given to verify the theoretical findings.
\\


\pagenumbering{arabic}

\bigskip

\section{Introduction}\label{sec:introduction}

In this paper, we apply the local discontinuous Galerkin (LDG) method on overlapping meshes provided in \cite{Overlap1} for the convection-diffusion equations
\begin{equation} \label{heat}
u_t+f(u)_x=(a^2(u)u_x)_x,\:\:\:\:x\in[0,1],\:\:t>0,
\end{equation}
as well as its two dimensional version. We assume that $a(u)\geq0$.

In 1973, Reed and Hill first introduced the discontinuous Galerkin (DG) method in the framework of neutron linear transportation \cite{1stDG}. This method gained even greater popularity for good stability, high-order accuracy, and flexibility on h-p adaptivity and complex geometry. Subsequently, Cockburn et. al. proposed in a series of papers \cite{Cockburn1, Cockburn2, Cockburn3, Cockburn4} the Runge-Kutta discontinuous Galerkin (RKDG) methods for hyperbolic conservation laws. Later, in \cite{Cockburn5}, Cockburn and Shu introduced the LDG method for convection-diffusion equations motivated by successful solving compressible Navier-Stokes equations in \cite{Bassi}.

%Recently, The MPP technique becomes popular to study in various equations. For example, genuinely MPP high-order DG schemes for scalar conservation laws and two-dimensional incompressible flows in vorticity-streamfunction formulation have been constructed in \cite{3MPPconvect}. Then, the researchers used the technique to other hyperbolic systems, such as pressureless Euler equations \cite{MPPHyper2}, Extended MHD equations \cite{MPPHyper3}, relativistic hydrodynamics \cite{MPPHyper1}. The second-order MPP discontinuous Galerkin method for parabolic equations were constructed but not easy to do \cite{2MPPpara}.  Later, the flux limiter was used to improved the construction stated in \cite{Fluxlimit1, Fluxlimit2}. For the third-order MPP, the convection diffusion equation with direct DG method was introduced in \cite{3MPPDDG}. However, we need to add two penalty terms that make the scheme was not easy to construct. Recently, Du and Yang constructed MPP third-order LDG method on overlapping meshes \cite{3MPPOverlap1}.

As in traditional LDG method, we introduce an auxiliary variable $p=A(u)_x$ with $A(u)=\int^u a(s)\ ds$ to represent the derivative of the primary variable $u$, and rewrite \eqref{heat} into the following system of first order equations
\begin{align}\label{heatorder0}
\begin{cases}
u_t+f(u)_x=(a(u)p)_x,\\
p=A(u)_x.
\end{cases}
\end{align}
Then we can solve $u$ and $p$ on the same mesh by using the DG method. The LDG method shares all the nice features of the DG methods for hyperbolic equations, and it becomes one of the most popular numerical methods for solving convection-diffusion equations. However, due to the discontinuity nature of the numerical approximations, it may not be easy to construct and analyze the scheme for some specials convection-diffusion equations. For example, the convection terms of chemotaxis model \cite{Chemo1, Chemo2} and miscible displacements in porous media \cite{Mixmethod1, Mixmethod2} are products of one of the primary variable and the derivative of another one. Therefore, the upwind flux for the convection term may not be easy to obtain. One of the alternatives is to use other methods, such as mixed finite element method, to obtain continuous approximations of the derivatives, see e.g. \cite{Guo2014}. A more general idea is to use the Lax-Friedrichs flux, see e.g. \cite{Guo2017,Li2017,Yu2017} for the error estimates for miscible displacements and chemotaxis models. The main technique is to use the diffusion term to control the convection term \cite{Wang2015,Wang2016,Wang2016b}. Moreover, to make the numerical solutions to be physically relevant, we have to add a sufficiently large penalty which depends on the numerical approximations of the derivatives of the primary variables \cite{Guo2017b,Li2017,Chuenjarern2019}. Another possible way is to construct flux-free schemes, such as the central discontinuous Galerkin (CDG) method \cite{CDG} and the staggered discontinuous Galerkin (SDG) method \cite{SDG}. However, the CDG scheme doubles the computational cost as we have to solve each equation in \eqref{heatorder0} on both the primary and dual meshes twice and it is not easy to apply limiters in SDG method because it requires partial continuity of the numerical approximations.

Recently, one of the authors in this paper introduced a new LDG method on overlapping meshes \cite{Overlap1} by solving $u$ and $p$ on primitive and dual meshes, respectively, hence $p$ is continuous across the interfaces on the primitive mesh. The scheme is proved to be stable under the $L^2$-norm and can be used to construct third-order maximum-principle-preserving schemes \cite{Du2019}. However, in some special cases, it may not enjoy the optimal convergence rates. The suboptimal convergence rate can be observed numerically if all the following three conditions are satisfied: (1) Odd order polynomials are used in the finite element space, (2) The dual mesh generated by connecting the midpoints of the primitive mesh, (3) No penalty is added to the numerical scheme. If one of the conditions is violated, the convergence rate will turn out to be optimal. Later, in \cite{FouriernewLDG}, we used Fourier analysis to explicitly write out the error between the numerical and exact solutions and verify the optimal convergence rate for linear parabolic equations with periodic boundary conditions in one space dimension. Moreover, we also found out some superconvergence points that may depend on the perturbation constant in the construction of the dual mesh.

Both works given above are for problems with periodic boundary conditions. To implement the scheme, we need to combine the two boundary cells at the boundaries into one and find a polynomial approximation on the new cell. It is impossible to do that for general Dirichlet and Neumann boundary conditions, which are more realistic in practice, see e.g. \cite{Mixmethod1, Mixmethod2,Li2017}. In this paper, we will discuss the stability and error estimates of the new LDG methods for problems with Neumann and Dirichlet boundary conditions. The difficulty for the Neumann and Dirichlet boundary conditions is how to deal with the boundary cells of the dual mesh since two boundary cells cannot be combined. One possible way is to leave two boundary cells after generating the dual mesh, and introduce suitable numerical fluxes are the boundaries. For simplicity of presentation, we only demonstrate the proof for nonlinear parabolic equations
\begin{align}\label{heatorder1}
\begin{cases}
u_t=(a(u)p)_x,\\
p=A(u)_x,
\end{cases}
\end{align}
where $\displaystyle A(u)=\int^u a(t)\:dt$. The extension to general nonlinear convection-diffusion equations can be obtained following \cite{Overlap1}, hence we only demonstrate the results without proof.

The rest of the paper is organized as follows: we first discuss the LDG scheme on overlapping mesh in Section \ref{sec:overlappingmesh}. In Section \ref{sec:stablility}, we demonstrate the stability analysis of the scheme for the Neumann and Dirichlet boundary conditions. The error estimate will be provided in Section \ref{sec:errorestimate}. The extension to problems in two space dimensions will be discussed in Section \ref{sec:2Dscheme}. In Section \ref{sec:Exmaple}, the numerical experiments will be given to demonstrate the accuracy of the scheme on non-periodic boundary conditions. We will end in Section \ref{sec:conclusion} with concluding remarks.

\section{Preliminary}\label{sec:overlappingmesh}
In this section, we proceed to demonstrate the new LDG method for solving the one-dimensional diffusion equation \eqref{heatorder1} on overlapping meshes with three different boundary conditions, i.e. periodic, Neumann and Dirichlet boundary conditions.
\begin{figure}
\centering
\begin{tikzpicture}

\draw (0,0) -- (8,0);
\draw (0,0) -- (0,0.35);
\draw (4,0) -- (4,0.35);
\draw (8,0) -- (8,0.35);
\draw (2.5,-2) -- (6.5,-2);
\draw (2.5,-2) -- (2.5,-1.65);
\draw (6.5,-2) -- (6.5,-1.65);
\draw [dotted] (2.5,-1.65) -- (2.5,0);
\draw [dotted] (6.5,-1.65) -- (6.5,0);
\draw (2,0.4) node[above] {$I_i$};
\draw (6,0.4) node[above] {$I_{i+1}$};
\draw (4.5,-1.6) node[above] {$P_{i+1/2}$};
\draw (0,0) node[below] {$x_{i-1/2}$};
\draw (4,0) node[below] {$x_{i+1/2}$};
\draw (8,0) node[below] {$x_{i+3/2}$};
\draw (2.5,-2) node[below] {$\tilde{x}_{i}$};
\draw (6.5,-2) node[below] {$\tilde{x}_{i+1}$};

\end{tikzpicture}
\caption {Overlapping meshes}
\label{fig:overlappingmesh}
\end{figure}
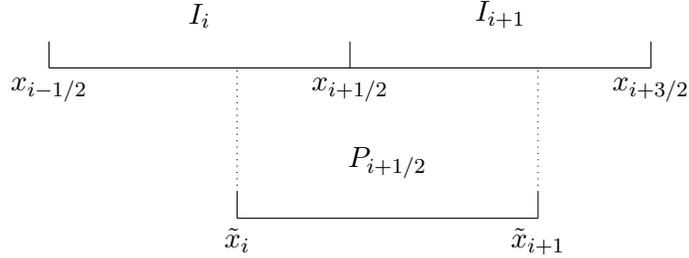
\subsection{Overlapping meshes}
The new LDG method solves the variables $u$ and $p$ on two different meshes as shown in Figure \ref{fig:overlappingmesh}. First, we define the primitive mesh on which the primary variable $u$ is solved. We give a partition of the computational domain $\Omega=[0,1]$ as $0=x_{\frac12}<x_{\frac32}<\cdots<x_{N+\frac12}=1$, and denote the $i-$th cell as
\[I_i=[x_{i-\frac{1}{2}},x_{i+\frac{1}{2}}],\:\:\:\:i=1,...,N.\]
Moreover, we denote
\[\Delta x_i=x_{i+\frac{1}{2}}-x_{i-\frac{1}{2}},\:\:\:\:x_i=\frac{x_{i+\frac{1}{2}}+x_{i-\frac{1}{2}}}{2}\]
as the cell length and the cell center of $I_i$, respectively, and define $\displaystyle \Delta x =\max_i\Delta x_i$.

We now demonstrate how to create the dual mesh, namely the P-mesh in \cite{Overlap1}, for solving the auxiliary variable $p$ for problems with periodic boundary conditions. We choose a point $\tilde{x}_i$ given as
\begin{equation}\label{pointondual}
\tilde{x}_i=x_i+\frac{\Delta x_i}{2}\xi_{0,i},\:\:\:\:\xi_{0,i}\in[-1,1],\:\:\:i=1,...,N.
\end{equation}
It is easy to see that $\tilde{x}_i\in [x_{i-\frac{1}{2}},x_{i+\frac{1}{2}}]$. Define
\[P_{i-\frac{1}{2}}=[\tilde{x}_{i-1},\tilde{x}_i],\:\:\:\:i=1,...,N,\]
as the ($i-\frac{1}{2})$-th cell of the dual mesh where we denote $\tilde{x}_0=\tilde{x}_N-1$.
Moreover,
\[\Delta \tilde{x}_{i-\frac{1}{2}}=\tilde{x}_{i}-\tilde{x}_{i-1},\:\:\:\:\tilde{x}_{i-\frac{1}{2}}=\frac{\tilde{x}_{i}+\tilde{x}_{i-1}}{2}\]
stand for the cell length and the cell center of $P_{i-\frac{1}{2}}$, respectively. The P-mesh will consist of all $P$ cells. For periodic boundary conditions, we can also define $P_{\frac{1}{2}}=[0,\tilde{x}_1]\cup[\tilde{x}_N,1]$ as stated in \cite{Overlap1}, see Figure  \ref{fig:meshPeriodic}.

Next, we demonstrate how to deal with the boundary cells on the P-mesh for problems with Neumann and Dirichlet boundary conditions. We leave the left and the right boundary cells and define $P_{\frac{1}{2}}=[0,\tilde{x}_1]$ and $P_{N+\frac{1}{2}}=[\tilde{x}_N,1]$, see Figure  \ref{fig:meshleaving}. This mesh is called the L-mesh. The other way is to combine the boundary cells with their neighbour as $P_{\frac{3}{2}}=[0,\tilde{x}_2]$ and $P_{N-\frac{1}{2}}=[\tilde{x}_{N-1},1]$, see Figure  \ref{fig:meshcombining}. This mesh is called C-mesh.

If we take $\xi_{0,i}$ as a constant independent of $i$, it is easy to check that
\[\min\{\Delta x_{i-1},\Delta x_i\}\leq \Delta \tilde{x}_{i-\frac{1}{2}}\leq \max\{\Delta x_{i-1},\Delta x_i\},\]
and hence we have $\max_i\Delta \tilde{x}_{i-\frac{1}{2}}\leq \Delta x.$
\begin{figure}
\centering
\begin{tikzpicture}
\draw (-6,0) -- (-2,0);
\draw (0,0) -- (4,0);
\draw (6,0) -- (10,0);
\draw [dash dot](-2,0) -- (0,0);
\draw [dash dot](4,0) -- (6,0);
\draw (-6,0) -- (-6,0.35);
\draw (-4,0) -- (-4,0.35);
\draw (-2,0) -- (-2,0.35);
\draw (0,0) -- (0,0.35);
\draw (2,0) -- (2,0.35);
\draw (4,0) -- (4,0.35);
\draw (6,0) -- (6,0.35);
\draw (8,0) -- (8,0.35);
\draw (10,0) -- (10,0.35);
\draw (-6,-2) -- (-2.65,-2);
\draw (1.35,-2) -- (3.35,-2);
\draw (7.35,-2) -- (10,-2);
\draw [dash dot](-2.65,-2) -- (1.35,-2);
\draw [dash dot](3.35,-2) -- (7.35,-2);
\draw (-6,-2) -- (-6,-1.65);
\draw (-4.65,-2) -- (-4.65,-1.65);
\draw (-2.65,-2) -- (-2.65,-1.65);
\draw (1.35,-2) -- (1.35,-1.65);
\draw (3.35,-2) -- (3.35,-1.65);
\draw (7.35,-2) -- (7.35,-1.65);
\draw (9.35,-2) -- (9.35,-1.65);
\draw (10,-2) -- (10,-1.65);
\draw [dotted] (-4.65,-1.65) -- (-4.65,0);
\draw [dotted] (-2.65,-1.65) -- (-2.65,0);
\draw [dotted] (1.35,-1.65) -- (1.35,0);
\draw [dotted] (3.35,-1.65) -- (3.35,0);
\draw [dotted] (7.35,-1.65) -- (7.35,0);
\draw [dotted] (9.35,-1.65) -- (9.35,0);
\draw (1,0.4) node[above] {$I_i$};
\draw (3,0.4) node[above] {$I_{i+1}$};
\draw (-5,0.4) node[above] {$I_1$};
\draw (-3,0.4) node[above] {$I_2$};
\draw (7,0.4) node[above] {$I_{N-1}$};
\draw (9,0.4) node[above] {$I_{N}$};
\draw (2.35,-1.6) node[above] {$P_{i+1/2}$};
\draw (-3.65,-1.6) node[above] {$P_{3/2}$};
\draw (-5.65,-1.6) node[above] {$P_{1/2}$};
\draw (8.35,-1.6) node[above] {$P_{N-1/2}$};
\draw (10,-1.6) node[above] {$P_{1/2}$};
\draw (-6,0) node[below] {$x_{1/2}$};
\draw (-4,0) node[below] {$x_{3/2}$};
\draw (-2,0) node[below] {$x_{5/2}$};
\draw (0,0) node[below] {$x_{i-1/2}$};
\draw (2,0) node[below] {$x_{i+1/2}$};
\draw (4,0) node[below] {$x_{i+3/2}$};
\draw (6,0) node[below] {$x_{N-3/2}$};
\draw (8,0) node[below] {$x_{N-1/2}$};
\draw (10,0) node[below] {$x_{N+1/2}$};
\draw (-6,-2) node[below] {$\tilde{x}_{0}$};
\draw (-4.65,-2) node[below] {$\tilde{x}_{1}$};
\draw (-2.65,-2) node[below] {$\tilde{x}_{2}$};
\draw (1.35,-2) node[below] {$\tilde{x}_{i}$};
\draw (3.35,-2) node[below] {$\tilde{x}_{i+1}$};
\draw (7.35,-2) node[below] {$\tilde{x}_{N-1}$};
\draw (9.35,-2) node[below] {$\tilde{x}_{N}$};
\draw (10.25,-2) node[below] {$\tilde{x}_{0}$};
\end{tikzpicture}
\caption {Overlapping meshes for periodic boundary conditions.}
\label{fig:meshPeriodic}

\begin{tikzpicture}
\draw (-6,0) -- (-2,0);
\draw (0,0) -- (4,0);
\draw (6,0) -- (10,0);
\draw [dash dot](-2,0) -- (0,0);
\draw [dash dot](4,0) -- (6,0);
\draw (-6,0) -- (-6,0.35);
\draw (-4,0) -- (-4,0.35);
\draw (-2,0) -- (-2,0.35);
\draw (0,0) -- (0,0.35);
\draw (2,0) -- (2,0.35);
\draw (4,0) -- (4,0.35);
\draw (6,0) -- (6,0.35);
\draw (8,0) -- (8,0.35);
\draw (10,0) -- (10,0.35);
\draw (-6,-2) -- (-2.65,-2);
\draw (1.35,-2) -- (3.35,-2);
\draw (7.35,-2) -- (10,-2);
\draw [dash dot](-2.65,-2) -- (1.35,-2);
\draw [dash dot](3.35,-2) -- (7.35,-2);
\draw (-6,-2) -- (-6,-1.65);
\draw (-4.65,-2) -- (-4.65,-1.65);
\draw (-2.65,-2) -- (-2.65,-1.65);
\draw (1.35,-2) -- (1.35,-1.65);
\draw (3.35,-2) -- (3.35,-1.65);
\draw (7.35,-2) -- (7.35,-1.65);
\draw (9.35,-2) -- (9.35,-1.65);
\draw (10,-2) -- (10,-1.65);
\draw [dotted] (-4.65,-1.65) -- (-4.65,0);
\draw [dotted] (-2.65,-1.65) -- (-2.65,0);
\draw [dotted] (1.35,-1.65) -- (1.35,0);
\draw [dotted] (3.35,-1.65) -- (3.35,0);
\draw [dotted] (7.35,-1.65) -- (7.35,0);
\draw [dotted] (9.35,-1.65) -- (9.35,0);
\draw (1,0.4) node[above] {$I_i$};
\draw (3,0.4) node[above] {$I_{i+1}$};
\draw (-5,0.4) node[above] {$I_1$};
\draw (-3,0.4) node[above] {$I_2$};
\draw (7,0.4) node[above] {$I_{N-1}$};
\draw (9,0.4) node[above] {$I_{N}$};
\draw (2.35,-1.6) node[above] {$P_{i+1/2}$};
\draw (-3.65,-1.6) node[above] {$P_{3/2}$};
\draw (-5.65,-1.6) node[above] {$P_{1/2}$};
\draw (8.35,-1.6) node[above] {$P_{N-1/2}$};
\draw (10,-1.6) node[above] {$P_{N+1/2}$};
\draw (-6,0) node[below] {$x_{1/2}$};
\draw (-4,0) node[below] {$x_{3/2}$};
\draw (-2,0) node[below] {$x_{5/2}$};
\draw (0,0) node[below] {$x_{i-1/2}$};
\draw (2,0) node[below] {$x_{i+1/2}$};
\draw (4,0) node[below] {$x_{i+3/2}$};
\draw (6,0) node[below] {$x_{N-3/2}$};
\draw (8,0) node[below] {$x_{N-1/2}$};
\draw (10,0) node[below] {$x_{N+1/2}$};
\draw (-6,-2) node[below] {$\tilde{x}_{0}$};
\draw (-4.65,-2) node[below] {$\tilde{x}_{1}$};
\draw (-2.65,-2) node[below] {$\tilde{x}_{2}$};
\draw (1.35,-2) node[below] {$\tilde{x}_{i}$};
\draw (3.35,-2) node[below] {$\tilde{x}_{i+1}$};
\draw (7.35,-2) node[below] {$\tilde{x}_{N-1}$};
\draw (9.35,-2) node[below] {$\tilde{x}_{N}$};
\draw (10.25,-2) node[below] {$\tilde{x}_{N+1}$};
\end{tikzpicture}
\caption {Overlapping mesh 1, L-mesh.}
\label{fig:meshleaving}

\centering
\begin{tikzpicture}
\draw (-6,0) -- (-2,0);
\draw (0,0) -- (4,0);
\draw (6,0) -- (10,0);
\draw [dash dot](-2,0) -- (0,0);
\draw [dash dot](4,0) -- (6,0);
\draw (-6,0) -- (-6,0.35);
\draw (-4,0) -- (-4,0.35);
\draw (-2,0) -- (-2,0.35);
\draw (0,0) -- (0,0.35);
\draw (2,0) -- (2,0.35);
\draw (4,0) -- (4,0.35);
\draw (6,0) -- (6,0.35);
\draw (8,0) -- (8,0.35);
\draw (10,0) -- (10,0.35);
\draw (-6,-2) -- (-2.65,-2);
\draw (1.35,-2) -- (3.35,-2);
\draw (7.35,-2) -- (10,-2);
\draw [dash dot](-2.65,-2) -- (1.35,-2);
\draw [dash dot](3.35,-2) -- (7.35,-2);
\draw (-6,-2) -- (-6,-1.65);
\draw (-2.65,-2) -- (-2.65,-1.65);
\draw (1.35,-2) -- (1.35,-1.65);
\draw (3.35,-2) -- (3.35,-1.65);
\draw (7.35,-2) -- (7.35,-1.65);
\draw (10,-2) -- (10,-1.65);
\draw [dotted] (-2.65,-1.65) -- (-2.65,0);
\draw [dotted] (1.35,-1.65) -- (1.35,0);
\draw [dotted] (3.35,-1.65) -- (3.35,0);
\draw [dotted] (7.35,-1.65) -- (7.35,0);
\draw (1,0.4) node[above] {$I_i$};
\draw (3,0.4) node[above] {$I_{i+1}$};
\draw (-5,0.4) node[above] {$I_1$};
\draw (-3,0.4) node[above] {$I_2$};
\draw (7,0.4) node[above] {$I_{N-1}$};
\draw (9,0.4) node[above] {$I_{N}$};
\draw (2.35,-1.6) node[above] {$P_{i+1/2}$};
\draw (-4.35,-1.6) node[above] {$P_{3/2}$};
\draw (8.65,-1.6) node[above] {$P_{N-1/2}$};
\draw (-6,0) node[below] {$x_{1/2}$};
\draw (-4,0) node[below] {$x_{3/2}$};
\draw (-2,0) node[below] {$x_{5/2}$};
\draw (0,0) node[below] {$x_{i-1/2}$};
\draw (2,0) node[below] {$x_{i+1/2}$};
\draw (4,0) node[below] {$x_{i+3/2}$};
\draw (6,0) node[below] {$x_{N-3/2}$};
\draw (8,0) node[below] {$x_{N-1/2}$};
\draw (10,0) node[below] {$x_{N+1/2}$};
\draw (-6,-2) node[below] {$\tilde{x}_{1}$};
\draw (-2.65,-2) node[below] {$\tilde{x}_{2}$};
\draw (1.35,-2) node[below] {$\tilde{x}_{i}$};
\draw (3.35,-2) node[below] {$\tilde{x}_{i+1}$};
\draw (7.35,-2) node[below] {$\tilde{x}_{N-1}$};
\draw (10.25,-2) node[below] {$\tilde{x}_{N}$};
\end{tikzpicture}
\caption {Overlapping mesh2, C-mesh.}
\label{fig:meshcombining}
\end{figure}
\subsection{Norms}
In this section, we define some norms that will be used throughout the paper.

For any interval $I$, we define $||u||_I$ and $||u||_{\infty,I}$ to be the standard $L^2$- and $L^\infty$- norms of $u$ on $I$, respectively. For any natural number $\ell$, we consider the norm of Sobolev space $H^\ell(I)$, defined by
\[||u||_{\ell,I}=\left\{ \sum_{0\leq\beta\leq\ell}\left\Vert\frac{\partial^\beta u}{\partial x^\beta}\right\Vert^2_I \right\}^{\frac{1}{2}}.\]
For convenience, if $I$ is the whole computational domain, then the corresponding subscript will be omitted.

Moreover, for any $u\in C(I_i)$, we define
\[||u||_{\Gamma_i}=|u^-_{i+\frac{1}{2}}|+|u^+_{i-\frac{1}{2}}|.\]
Similarly,  for any $u\in C(P_{i-\frac{1}{2}})$, we define
\[||u||_{\Gamma_{i-\frac{1}{2}}}=|u^-_{i-1}|+|u^+_{i}|.\]
\subsection{Numerical schemes with Neumann and Dirichlet boundary conditions }
In this section, we consider the LDG method for the following nonlinear diffusion equation  \begin{equation}\label{LDGequation}
\begin{cases}
u_t=(a(u)p)_x,\:\:\: x\in[0,1],\:\:t>0,\\
p=A(u)_x,\:\:\: x\in[0,1],\\
\end{cases}
\end{equation}
subject to the following boundary conditions
\begin{enumerate}
\item[] \textit{Neumann boundary condition}:
\begin{equation}\label{NeumannBC}
u_x(0,t)=u_x(1,t)=0.
\end{equation}
\item[] \textit{Dirichlet boundary condition}:
\begin{equation}\label{DirichletBC}
u(0,t)=u(1,t)=0.
\end{equation}
\end{enumerate}
We define the finite element spaces to be
\begin{align*}
V_h&=\{u_h: u_h|_{I_i}\in P^k(I_i), i=1,\ldots,N\},\\
P^L_h&=\{p_h: p_h|_{P_{i-\frac{1}{2}}}\in P^k(P_{i-\frac{1}{2}}), i=1,\ldots,N+1\},\\
P^C_h&=\{p_h: p_h|_{P_{i-\frac{1}{2}}}\in P^k(P_{i-\frac{1}{2}}), i=2,\ldots,N\},
\end{align*}
where $P^L_h$ and $P^C_h$ are finite element spaces for the L-mesh and C-mesh, respectively, and $P^k$ is the space of polynomials of degree up to $k$.

Multiplying \eqref{LDGequation} with the test functions and using the integration by parts, we obtain
\begin{align} \label{schemeu}
\int_{I_i}(u_h)_tv dx&=-\int_{I_i}a(u_h)p_hv_x dx+\hat{a}_{i+\frac{1}{2}}\widehat{p_h}_{i+\frac{1}{2}}v^-_{i+\frac{1}{2}}-\hat{a}_{i-\frac{1}{2}}\widehat{p_h}_{i-\frac{1}{2}}v^+_{i-\frac{1}{2}},\\
\int_{P_{i-\frac{1}{2}}}p_hw dx&=-\int_{P_{i-\frac{1}{2}}}A(u_h)w_x dx+ A(u_h(\tilde{x}_{i}))w^-_{i}-A(u_h(\tilde{x}_{i-1}))w^+_{i-1}. \label{schemep}
\end{align}
Then the LDG method on overlapping meshes is defined as follows:
\begin{itemize}
\item L-mesh:
find $(u_h,p_h)\in V_h\times P^L_h$, such that for any test functions $(v,w)\in V_h\times P^L_h$ we have \eqref{schemeu} and \eqref{schemep},
\item C-mesh:
find $(u_h,p_h)\in V_h\times P^C_h$, such that for any test functions $(v,w)\in V_h\times P^C_h$ we have \eqref{schemeu} and \eqref{schemep},
\end{itemize}
 where $v^-_{i+\frac{1}{2}}=v^-(x_{i+\frac{1}{2}})$ and $w^-_i=w^-(\tilde{x}_i)$. Likewise for $v^+_{i-\frac{1}{2}}$ and $w^+_{i-1}$. % Due to the boundary conditions we also define
%\begin{enumerate}
%\item[] \textit{Neumann boundary condition}: $v^-_{\frac{1}{2}}=v^+_{\frac{1}{2}}$ and $v^+_{N+\frac{1}{2}}=v^-_{N+\frac{1}{2}}$.
%\item[] \textit{Dirichlet boundary condition}: $v^-_{\frac{1}{2}}=0$ and $v^+_{N+\frac{1}{2}}=0$.
%\end{enumerate}

We denote the jump of the function $s$ across the cell interface $x=x_{i-\frac{1}{2}}$ as $[s]_{i-\frac{1}{2}}=s^+_{i-\frac{1}{2}}-s^-_{i-\frac{1}{2}}$. Similarly, $[w]_i=w^+_i-w^-_i$ denotes the jump of the function $w$ across the cell interface $x=\tilde{x}_i$ on the P-mesh. Due to the boundary conditions we also define the numerical fluxes of $u$ and $p$ on L-meshes as
\begin{enumerate}
\item[] \textit{Neumann boundary condition}: $u^-_{\frac{1}{2}}=u^+_{\frac{1}{2}}$,  $u^+_{N+\frac{1}{2}}=u^-_{N+\frac{1}{2}}$, $p^-_{1}=p^+_{1}$ and $p^+_{N}=p^-_{N}$.
\item[] \textit{Dirichlet boundary condition}: $u^-_{\frac{1}{2}}=0$, $u^+_{N+\frac{1}{2}}=0$, $p^-_{1}=0$ and $p^+_{N}=0$.
\end{enumerate}
The definition of the numerical fluxes at the boundary for C-mesh is similar. The numerical flux $\hat{a}$ at the point $x_{i-\frac{1}{2}}$ is defined as
\begin{equation}\label{aflux}
\hat{a}_{i-\frac{1}{2}}=\frac{[A(u_h)]_{i-\frac{1}{2}}}{[u_h]_{i-\frac{1}{2}}}.
\end{equation}
Also, we choose the numerical flux $\widehat{p_h}_{i-\frac{1}{2}}$ as the value of $p_h$ evaluated at $x=x_{i-\frac{1}{2}}$ with the penalty term
\begin{equation}\label{penalty}
\widehat{p_h}_{i-\frac{1}{2}}=p_h(x_{i-\frac{1}{2}})+\frac{\alpha_{i-\frac{1}{2}}}{\Delta\tilde{x}_{i-\frac{1}{2}}}[u_h]_{i-\frac{1}{2}}.
\end{equation}
Notice that $p_h$ is continuous at the interfaces of the primitive cells and hence $p_h(x_{i-\frac{1}{2}})$ is well defined.

Finally, we define
\begin{align}
H_u(u_h,p_h,v)&=-\sum_{i=1}^N\int_{I_i}a(u_h)p_hv_x dx+\sum_{i=1}^N\left(\hat{a}_{i+\frac{1}{2}}\hat{p}_{i+\frac{1}{2}}v^-_{i+\frac{1}{2}}-\hat{a}_{i-\frac{1}{2}}\hat{p}_{i-\frac{1}{2}}v^+_{i-\frac{1}{2}}\right),\label{HuNon}\\
H_p(u_h,w)&=-\sum_{i=1}^{N+1}\int_{P_{i-\frac{1}{2}}}A(u_h)w_x dx+\sum_{i=1}^{N+1}\left( A(u_h(\tilde{x}_{i}))w^-_{i}-A(u_h(\tilde{x}_{i-1}))w^+_{i-1}\right),\label{HpNon}
\end{align}
for L-mesh. For C-mesh, the two summations in \eqref{HpNon} are from $i=2$ to $N$
Then the LDG scheme can be rewritten as
\begin{align}
\int_\Omega (u_h)_tv dx&= H_u(u_h,p_h,v),\\
\int_\Omega p_hw dx&=H_p(u_h,w).
\end{align}

\section{Stability analysis}\label{sec:stablility}
In this section, we demonstrate the stability of the new LDG method on overlapping meshes with non-periodic boundary conditions.

\subsection{Neumann boundary condition}
In this subsection, the stability of the new LDG method on overlapping meshes with the Neumann boundary conditions will be demonstrated.
\begin{lemma}\label{lemma:1}
Suppose $H_u$ and $H_p$ are defined in \eqref{HuNon} and \eqref{HpNon}, respectively, and the dual mesh is given as either the L-mesh or the C-mesh, then we have
\begin{align}\label{Hu+HpNeumann}
H_u(u_h,p_h,v)+H_p(u_h,w)=-\sum_{i=2}^N\frac{[A(u_h)]_{i-\frac{1}{2}}}{[u_h]_{i-\frac{1}{2}}}\frac{\alpha_{i-\frac{1}{2}}}{\Delta\tilde{x}_{i-\frac{1}{2}}}[u_h]^2_{i-\frac{1}{2}}.
\end{align}
\end{lemma}
\begin{proof}
We only prove for L-mesh. The case for C-mesh is basically the same. Taking $w=p_h$ in \eqref{HpNon}, using integration by parts and applying the Neumann boundary conditions, we obtain
\begin{align} \notag
H_p(u_h,p_h)&=-\sum_{i=1}^{N+1}\int_{P_{i-\frac{1}{2}}}A(u_h)(p_h)_x dx+\sum_{i=1}^{N+1}\left(A(u_h(\tilde{x}_{i}))(p_h^-)_{i}-A(u_h(\tilde{x}_{i-1}))(p_h^+)_{i-1}\right)\\\notag
%&=-\int_{P_{\frac{1}{2}}}A(u_h)(p_h)_x dx+A(u_h(\tilde{x}_{1}))(p_h^-)_{1}-A(u_h(\tilde{x}_{0}))(p_h^+)_{0}\\\notag
%&\:\:\:\:\:-\sum_{i=2}^{N}\int_{P_{i-\frac{1}{2}}}A(u_h)(p_h)_x dx+\sum_{i=2}^{N}\left(A(u_h(\tilde{x}_{i}))(p_h^-)_{i}-A(u_h(\tilde{x}_{i-1}))(p_h^+)_{i-1}\right)\\\notag
%&\:\:\:\:\:-\int_{P_{N+\frac{1}{2}}}A(u_h)(p_h)_x dx+A(u_h(\tilde{x}_{N+1}))(p_h^-)_{N+1}-A(u_h(\tilde{x}_{N}))(p_h^+)_{N}\\ \notag
&=-\int_{P_{\frac{1}{2}}}A(u_h)(p_h)_x dx+A(u_h(\tilde{x}_{1}))(p_h^-)_{1}-A(u_h(\tilde{x}_{0}))(p_h^+)_{0}\\\notag
&\:\:\:\:\:-\sum_{i=2}^{N}\int^{x_{i-\frac{1}{2}}}_{\tilde{x}_{i-1}}A(u_h)(p_h)_x dx-\sum_{i=2}^{N}\int_{x_{i-\frac{1}{2}}}^{\tilde{x}_{i}}A(u_h)(p_h)_x dx\\\notag
&\:\:\:\:\:+\sum_{i=2}^{N}\left(A(u_h(\tilde{x}_{i}))(p_h^-)_{i}-A(u_h(\tilde{x}_{i-1}))(p_h^+)_{i-1}\right)\\\notag
&\:\:\:\:\:-\int_{P_{N+\frac{1}{2}}}A(u_h)(p_h)_x dx+A(u_h(\tilde{x}_{N+1}))(p_h^-)_{N+1}-A(u_h(\tilde{x}_{N}))(p_h^+)_{N}\\ \notag
&=\int_{P_{\frac{1}{2}}}a(u_h)(u_h)_xp_h dx\\\notag
&\:\:\:\:\:+\sum_{i=2}^{N}\int^{x_{i-\frac{1}{2}}}_{\tilde{x}_{i-1}}a(u_h)(u_h)_xp_h dx+\sum_{i=2}^{N}\int_{x_{i-\frac{1}{2}}}^{\tilde{x}_{i}}a(u_h)(u_h)_xp_h dx+\sum_{i=2}^{N}[A(u_h)]_{i-\frac{1}{2}}p_h(x_{i-\frac{1}{2}})\\\notag
&\:\:\:\:\:+\int_{P_{N+\frac{1}{2}}}a(u_h)(u_h)_xp_h dx\\
&=\sum_{i=1}^{N}\int_{I_i}a(u_h)(u_h)_xp_h dx+\sum_{i=2}^{N}[A(u_h)]_{i-\frac{1}{2}}p_h(x_{i-\frac{1}{2}}).\label{newHpNeumann}
\end{align}
Taking $v=u_h$ in \eqref{HuNon}, we obtain
\begin{align}\notag
H_u(u_h,p_h,v)&=-\sum_{i=1}^N\int_{I_i}p_ha(u_h)(u_h)_x dx-\sum_{i=2}^N\hat{a}_{i-\frac{1}{2}}\hat{p}_{i-\frac{1}{2}}[u_h]_{i-\frac{1}{2}}\\ \notag
&\:\:\:\:\:-\hat{a}_{\frac{1}{2}}\hat{p}_{\frac{1}{2}}(u_h)^+_{\frac{1}{2}}+\hat{a}_{N+\frac{1}{2}}\hat{p}_{N+\frac{1}{2}}(u_h)^-_{N+\frac{1}{2}}\\
&=-\sum_{i=1}^N\int_{I_i}p_ha(u_h)(u_h)_x dx-\sum_{i=2}^N\frac{[A(u_h)]_{i-\frac{1}{2}}}{[u_h]_{i-\frac{1}{2}}}\left(p_h(x_{i-\frac{1}{2}})+\frac{\alpha_{i-\frac{1}{2}}}{\Delta\tilde{x}_{i-\frac{1}{2}}}\right)[u_h]_{i-\frac{1}{2}}.\label{newHuNeumann}
\end{align}
where the last step, we obtain it by applying the Neumann boundary conditions.\\
Summing \eqref{newHpNeumann} and \eqref{newHuNeumann}, we have \eqref{Hu+HpNeumann} which further leads to the $L^2$ stability of the LDG method on overlapping meshes with Neumann boundary conditions.
\end{proof}

%\begin{remark}
%With a minor change of the proof, we can attain the proof for combining boundary cells treatment with the same result.
%\end{remark}
\subsection{Dirichlet boundary conditions}
In this subsection, we will describe the stability of the new LDG method on overlapping meshes for problems with Dirichlet boundary conditions.

With a minor change of the previous proof, we can obtain the following lemma.
\begin{lemma}\label{lemma:2}
Suppose $H_u$ and $H_p$ defined in \eqref{HuNon} and \eqref{HpNon}, respectively, then we have
\begin{align}\label{Hu+HpDirichlet}
H_u(u_h,p_h,v)+H_p(u_h,w)&=-\sum_{i=1}^{N+1}\frac{[A(u_h)]_{i-\frac{1}{2}}}{[u_h]_{i-\frac{1}{2}}}\frac{\alpha_{i-\frac{1}{2}}[u_h]^2_{i-\frac{1}{2}}}{\Delta\tilde{x}_{i-\frac{1}{2}}}.
\end{align}

\end{lemma}
The proof is very similar to that of Lemma \ref{lemma:1}, so we omit it and only demonstrate the result as the following theorem.
\begin{theorem}
The LDG method introduced \eqref{schemeu} and \eqref{schemep} with the boundary conditions \eqref{NeumannBC} and \eqref{DirichletBC} are stable and
\[\frac{1}{2}\frac{d}{dt}||u_h||^2+||p_h||^2\leq 0.\]
\end{theorem}

\section{Error estimates} \label{sec:errorestimate}
In this section, we demonstrate the error estimates. We will consider linear equations. Moreover, for simplicity, we will discuss the problem with the Neumann boundary condition on L-meshes only. For the other cases, we can apply the same procedure.

First, we use $e$ to denote the error between the exact and numerical solutions i.e. $e_u=u-u_h$ and $e_p=p-p_h$, then we can get the error equations from \eqref{schemeu} and \eqref{schemep} as
\begin{align}
\int_{I_i}(e_u)_tv dx&=-\int_{I_i}e_pv_x dx+\widehat{e_p}_{i+\frac{1}{2}}v^-_{i+\frac{1}{2}}-\widehat{e_p}_{i-\frac{1}{2}}v^+_{i-\frac{1}{2}},\label{errequNeumann}\\
\int_{P_{i-\frac{1}{2}}}e_pw dx&=-\int_{P_{i-\frac{1}{2}}}e_uw_x dx+ e_u(\tilde{x}_{i})w^-_{i}-e_u(\tilde{x}_{i-1})w^+_{i-1}. \label{erreqpNeumann}
\end{align}
Next, we introduce some basic properties of the finite element space that will be used.
\begin{lemma}\label{errlemma1}
 Assuming $u\in V_h$, there exists constant $C>0$ independent of $\Delta x$ and $u$ such that for $\beta \geq 1$
\[||\partial^\beta_x u||_{I_i}\leq C\Delta x_i^{-\beta}||u||_{I_i},\:\:\:\:||u||_{\Gamma_i}\leq C\Delta x_i^{-1/2}||u||_{I_i}.\]
Similarly, for any  $u\in P^L_h$, there exists constant $C>0$ independent of $\Delta \tilde{x}$ and $u$ such that for $\beta \geq 1$
\[||\partial^\beta_x u||_{P_{i-\frac{1}{2}}}\leq C\Delta \tilde{x}_{i-\frac{1}{2}}^{-\beta}||u||_{P_{i-\frac{1}{2}}},\:\:\:\:||u||_{\Gamma_{i-\frac{1}{2}}}\leq C\Delta \tilde{x}_{i-\frac{1}{2}}^{-1/2}||u||_{P_{i-\frac{1}{2}}}.\]
\end{lemma}
We also introduce the standard $L^2$ projections $P^1_k$ into $V_h$ and $P^2_k$ into $P^L_h$ by:
\[\int_{I_i}P^1_k uv dx=\int_{I_i} uv dx,\:\: \forall v\in P^k(I_i),\]
and
\[\int_{P_{i-\frac{1}{2}}}P^2_k uv dx=\int_{P_{i-\frac{1}{2}}} uv dx,\:\: \forall v\in P^k({P_{i-\frac{1}{2}}}),\]
respectively. By the scaling argument, we obtain the following lemma \cite{Penaltyterm}
\begin{lemma}\label{errlemma2}
Suppose the function $u(x)\in C^{k+1}(I_i)$, then there exists positive constant $C$ independent of $\Delta x$ and $u$, such that
\[||u-P^1_k u||_{I_i}+\Delta x_i||(u-P^1_k u)_x||_{I_i}+\Delta x^{1/2}_i||u-P^1_k u||_{\infty,{I_i}}\leq C\Delta x^{k+1}_i||u||_{k+1,I_i}.\]
Moreover, if $u(x)\in C^{k+1}(P_{i-\frac{1}{2}})$, then there exists positive constant $C$ independent of $\Delta \tilde{x}$ and $u$, such that
\[||u-P^2_k u||_{P_{i-\frac{1}{2}}}+\Delta \tilde{x}_{i-\frac{1}{2}}||(u-P^2_k u)_x||_{P_{i-\frac{1}{2}}}+\Delta \tilde{x}^{1/2}_{i-\frac{1}{2}}||u-P^2_k u||_{\infty,P_{i-\frac{1}{2}}}\leq C\Delta \tilde{x}^{k+1}_{i-\frac{1}{2}}||u||_{k+1,P_{i-\frac{1}{2}}},\]
where $||u||_{k+1,I}$ is the standard $H^{k+1}$-norm over the interval $I$.
\end{lemma}
As the general treatment of the finite element method, we split the errors into two terms as
\[e_u=\eta_u-\xi_u,\:\:\:\:e_p=\eta_p-\xi_p,\]
where
\[\eta_u=u-P^1_ku,\:\:\:\xi_u=u_h-P^1_ku,\:\:\:\eta_p=p-P^2_kp,\:\:\:\xi_p=p_h-P^2_kp.\]
With the above notations, the equations \eqref{errequNeumann} and \eqref{erreqpNeumann} can be rewritten as
\begin{align} \label{xierrequNeumann}
\int_{I_i}(\xi_u)_tv dx&=\int_{I_i}e_pv_x dx-\widehat{e_p}_{i+\frac{1}{2}}v^+_{i+\frac{1}{2}}-\widehat{e_p}_{i-\frac{1}{2}}v^+_{i-\frac{1}{2}}\\
\int_{P_{i-\frac{1}{2}}}\xi_pw dx&=\int_{P_{i-\frac{1}{2}}}e_uw_x dx- e_u(\tilde{x}_{i})w^-_{i}+e_u(\tilde{x}_{i-1})w^+_{i-1} \label{xierreqpNeumann}
\end{align}
Now, we can state the main theorem.
\begin{theorem}
Suppose the exact solution $u\in C^{k+2}(\Omega)$ and the finite element space is made up of piecewise polynomials of degree $k$. Moreover, the numerical solutions satisfy \eqref{schemeu} and \eqref{schemep}. Then the error between the numerical and exact solutions satisfies
\[||e_u||+\int_0^T||e_p|| dt \leq C\Delta x^k,\]
where $C$ is independent of $\Delta x$.
\begin{proof}
Sum up \eqref{xierrequNeumann} and \eqref{xierreqpNeumann} with $v=\xi_u$ and $w=\xi_p$, and then sum up over $i$ to obtain%. With \eqref{newHpNeumann} and \eqref{newHuNeumann}, we obtain
\begin{align} \notag
\frac{1}{2}\frac{d}{dt}||\xi_u||^2+||\xi_p||^2
&=\sum_{i=1}^N\int_{I_i}(\eta_p-\xi_p)(\xi_u)_x dx+\sum_{i=2}^N\left(\eta_{p_{i-\frac{1}{2}}}-\xi_{p_{i-\frac{1}{2}}}+\alpha_{i-\frac{1}{2}}\frac{[\eta_u-\xi_u]_{i-\frac{1}{2}}}{\Delta\tilde{x}_{i-\frac{1}{2}}}\right)[\xi_u]_{i-\frac{1}{2}}\\ \notag
&\:\:\:\:\:-\sum_{i=1}^N\int_{I_i}(\eta_u-\xi_u)(\xi_p)_x dx-\sum_{i=2}^N[\eta_u-\xi_u]_{i-\frac{1}{2}}\xi_p(x_{i-\frac{1}{2}})\\ \notag
&=\sum_{i=1}^N\int_{I_i}\eta_p(\xi_u)_x dx+\sum_{i=2}^N\left(\eta_{p_{i-\frac{1}{2}}}+\alpha_{i-\frac{1}{2}}\frac{[\eta_u]_{i-\frac{1}{2}}}{\Delta\tilde{x}_{i-\frac{1}{2}}}\right)[\xi_u]_{i-\frac{1}{2}}\\ \notag
&\:\:\:\:\:-\sum_{i=1}^N\int_{I_i}\eta_u(\xi_p)_x dx-\sum_{i=2}^N[\eta_u]_{i-\frac{1}{2}}\xi_p(x_{i-\frac{1}{2}})+H_u^N(\xi_u,\xi_p,\xi_u)+H_p^N(\xi_u,\xi_p)\\
&=R_1+R_2+R_3, \label{errorestimateterm}
\end{align}
where
\begin{align*}
R_1&=\sum_{i=1}^N\int_{I_i}\eta_p(\xi_u)_x dx-\sum_{i=1}^N\int_{I_i}\eta_u(\xi_p)_x dx,\\
R_2&=\sum_{i=2}^N\left(\alpha_{i-\frac{1}{2}}\frac{[\eta_u]_{i-\frac{1}{2}}}{\Delta\tilde{x}_{i-\frac{1}{2}}}\right)[\xi_u]_{i-\frac{1}{2}}+H_u^N(\xi_u,\xi_p,\xi_u)+H_p^N(\xi_u,\xi_p),\\
R_3&=\sum_{i=2}^N\eta_{p_{i-\frac{1}{2}}}[\xi_u]_{i-\frac{1}{2}}-\sum_{i=2}^N[\eta_u]_{i-\frac{1}{2}}\xi_p(x_{i-\frac{1}{2}}).
\end{align*}
Now we estimate $R_i$ where $i=1,2,3$ term by term.
\begin{align} \notag
R_1&\leq \sum_{i=1}^N\left(||\eta_p||_{I_i}||(\xi_u)_x||_{I_i}+||(\eta_u)_x||_{I_i}||\xi_p||_{I_i}\right) \\ \notag
&\leq \sum_{i=1}^N\left(||\eta_p||_{P_{i-\frac{1}{2}}\cup P_{i+\frac{1}{2}}}||\xi_u||_{I_i}+||(\eta_u)_x||_{I_i}||\xi_p||_{I_i}\right) \\ \notag
&\leq C\Delta x^k\sum_{i=1}^N\left[\left(||p||_{k+1,P_{i-\frac{1}{2}}}+||p||_{k+1,P_{i+\frac{1}{2}}}\right)||\xi_u||_{I_i}+||u||_{k+1,I_i}||\xi_p||_{I_i}\right]\\
&\leq C\Delta x^k \left(||\xi_u||+||\xi_p||\right), \label{R1}
\end{align}
where we applied the Cauchy-Schwarz inequality to the first step. In the second step, we used Lemmas \ref{errlemma1} and \ref{errlemma2}. Also, the Cauchy-Schwarz inequality was used again in the last step. Applying Lemma \ref{lemma:1}, we obtain the estimate of $R_2$
\begin{align} \notag
R_2&\leq \sum_{i=2}^N\frac{\alpha_{i-\frac{1}{2}}}{\Delta\tilde{x}_{i-\frac{1}{2}}}\left([\eta_u]_{i-\frac{1}{2}}[\xi_u]_{i-\frac{1}{2}}-[\xi_u]^2_{i-\frac{1}{2}}\right)\\ \notag
&\leq C \sum_{i=2}^N\frac{\alpha_{i-\frac{1}{2}}}{\Delta x}[\eta_u]^2_{i-\frac{1}{2}}\\ \notag
&\leq C \sum_{i=2}^N \alpha_{i-\frac{1}{2}}\Delta x^{2k}\left(||u||^2_{I_{i-1}}+||u||^2_{I_i}\right)\\
&\leq C\Delta x^{2k}, \label{R2}
\end{align}
where Lemma \ref{errlemma2} was applied in step 2. While steps 2 and 4 follow from direct computation. Finally, we estimate $R_3$.
\begin{align} \notag
R_3&\leq \sum_{i=2}^N||\eta_p||_{\infty,P_{i-\frac{1}{2}}}\left(||\xi_u||_{\Gamma_{i-1}}+||\xi_u||_{\Gamma_{i}}\right)+\sum_{i=2}^N\left(||\eta_u||_{\Gamma_{i-1}}+||\eta_u||_{\Gamma_{i}}\right)||\xi_p||_{\infty,I_{i}}\\ \notag
&\leq  C\Delta x^k \sum_{i=2}^N\left[||p||_{k+1,P_{i-\frac{1}{2}}}\left(||\xi_u||_{i-1}+||\xi_u||_{i}\right)+\left(||u||_{i-1}+||u||_{i}\right)||\xi_p||_{k+1,I_{i}}\right]\\
&\leq  C\Delta x^k \left(||p||_{k+1}||\xi_u||+||u||_{k+1}||\xi_p||\right),\label{R3}
\end{align}
where step 1 is straightforward, step 2 follows from Lemmas \ref{errlemma1} and \ref{errlemma2} and the last step we applied the Cauchy-Schwarz inequality. Substitute \eqref{R1}-\eqref{R3} into \eqref{errorestimateterm} to obtain
\[\frac{1}{2}\frac{d}{dt}||\xi_u||^2+||\xi_p||^2\leq C\Delta x^{2k}+C\Delta x^k(||\xi_u||+||\xi_p||),\]
which further yields
\[\frac{1}{2}\frac{d}{dt}||\xi_u||^2+||\xi_p||^2\leq C\Delta x^{2k}+||\xi_u||^2.\]
Finally, we can apply the Gronwall's inequality and complete the proof.
\end{proof}
\end{theorem}

\section{Numerical scheme for two-dimensional case}\label{sec:2Dscheme}
In this section, we will construct the scheme in two-dimensional case for the following problem over the domain $\Omega=[0,1]\times[0,1]$
\begin{equation}\label{2DLDGequation}
\begin{cases}
u_t=(a(u)p)_x+(b(u)q)_x,\\
p=A(u)_x,\\
q=B(u)_y.
\end{cases}
\end{equation}
where $\displaystyle A(u)=\int^u a(t)\:dt$ and $\displaystyle B(u)=\int^u b(t)\:dt.$ We consider the following boundary conditions
\begin{enumerate}
\item[] \textit{Neumann boundary condition}:
\begin{equation}\label{2DNeumannBC}
u_x(0,y,t)=u_x(1,y,t)=u_y(x,0,t)=u_y(x,1,t)=0.
\end{equation}
\item[] \textit{Dirichlet boundary condition}:
\begin{equation}\label{2DDirichletBC}
u(0,y,t)=u(1,y,t)=u(x,0,t)=u(x,1,t)=0.
\end{equation}
\end{enumerate}
First, we give a rectangular decomposition of $\Omega$ which is the primitive mesh for the primary variable $u$. Let $0=x_{\frac12}<x_{\frac32}<\cdots<x_{N_x+\frac12}=1$, and  $0=y_{\frac12}<y_{\frac32}<\cdots<y_{N_y+\frac12}=1$ be grid points in $x$ and $y$ directions, respectively, and denote  the $i,j-th$ cell as
\[I_{ij}=I_i\times J_j=[x_{i-\frac{1}{2}},x_{i+\frac{1}{2}}]\times[y_{j-\frac{1}{2}},y_{j+\frac{1}{2}}]\]
for all $i=1,\ldots,N_x$ and $j=1,\ldots,N_y$. Also, we denote
\[\Delta x_i=x_{i+\frac{1}{2}}-x_{i-\frac{1}{2}},\:\:\:\:x_i=\frac{x_{i+\frac{1}{2}}+x_{i-\frac{1}{2}}}{2},\:\:\:\:\Delta y_i=y_{j+\frac{1}{2}}-y_{j-\frac{1}{2}},\:\:\:\:y_j=\frac{y_{j+\frac{1}{2}}+y_{j-\frac{1}{2}}}{2},\]
and
\[\Delta x=\max_i\Delta x_i,\:\:\:\:\Delta y=\max_j\Delta y_j,\:\:\:\:h=\max\{\Delta x, \Delta y\}. \]
Next, we define the P-mesh and Q-mesh for solving the auxiliary variable $p$ and $q$. we choose $\tilde{x}_i$ given as
\begin{equation}\label{2Dpointondualx}
\tilde{x}_i=x_i+\frac{\Delta x_i}{2}\xi_{0},\:\:\:\:\xi_{0}\in[-1,1]
\end{equation}
to define the P-mesh as
\[P_{i-\frac{1}{2},j}=[\tilde{x}_{i-1},\tilde{x}_i]\times[y_{j-\frac{1}{2}},y_{j+\frac{1}{2}}],\:\:\:i=1,...,N_x,\]
with $\tilde{x}_0=\tilde{x}_{N_x-1}$. Similarly, we pick a point $\tilde{y}_j$ given as
\begin{equation}\label{2Dpointondualy}
\tilde{y}_j=y_j+\frac{\Delta y_j}{2}\eta_{0},\:\:\:\:\eta_{0}\in[-1,1],
\end{equation}
and define Q-mesh as
\[Q_{i,j-\frac{1}{2}}=[x_{i-\frac{1}{2}},x_{i+\frac{1}{2}}]\times[\tilde{y}_{j-1},\tilde{y}_j],\:\:\:j=1,...,N_y,\]
with $\tilde{y}_0=\tilde{y}_{N_y-1}$.
Similar to the problem in one space dimension, we need to deal with the boundary cells for problems with the Neumann and Dirichlet boundary conditions. We can leave the left and right boundary cells as $P_{\frac{1}{2},j}=[0,\tilde{x}_1]\times J_j$, $P_{N_x+\frac{1}{2},j}=[\tilde{x}_{N_x},1]\times J_j$, $Q_{i,\frac{1}{2}}=I_i\times[0,\tilde{y}_1]$ and $Q_{i,N_y+\frac{1}{2}}=I_i\times[\tilde{y}_{N_y},1]$ . This mesh is called the L-mesh. The other way is to combine the boundary cells with their neighbour as $P_{\frac{3}{2},j}=[0,\tilde{x}_2]\times J_j, P_{N_x-\frac{1}{2}}=[\tilde{x}_{N_x-1},1]\times J_j, Q_{i,\frac{3}{2}}=I_i\times[0,\tilde{x}_2]$  and $Q_{N_y-\frac{1}{2}}=I_i\times[\tilde{y}_{N_y-1},1]$. This mesh is called C-mesh.

We define the finite element spaces for the L-mesh, $P^L_h$ and $Q^L_h$, and for the  C-mesh, $P^C_h$ and $Q^C_h$, to be
\begin{align*}
V_h&=\{u_h: u_h|_{I_ij}\in P^k(I_ij), i=1,\ldots,N_x, j=1,\ldots,N_y\},\\
P^L_h&=\{p_h: p_h|_{P_{i-\frac{1}{2}},j}\in P^k(P_{i-\frac{1}{2},j}), i=1,\ldots,N_x+1, j=1,\ldots,N_y\},\\
P^C_h&=\{p_h: p_h|_{P_{i-\frac{1}{2}},j}\in P^k(P_{i-\frac{1}{2},j}), i=2,\ldots,N_x, j=1,\ldots,N_y\},\\
Q^L_h&=\{p_h: p_h|_{Q_{i,j-\frac{1}{2}}}\in P^k(Q_{i,j-\frac{1}{2}}), i=1,\ldots,N_x, j=2,\ldots,N_y+1\},\\
Q^C_h&=\{p_h: p_h|_{Q_{i,j-\frac{1}{2}}}\in P^k(Q_{i,j-\frac{1}{2}}), i=1,\ldots,N_x, j=2,\ldots,N_y\},
\end{align*}
where $P^k$ is the space of polynomials of degree up to $k$.

Now, we can introduce the LDG method on overlapping mesh for \eqref{2DLDGequation}. Multiplying the test functions and using the integration by parts, we obtain
\begin{align} \notag
\int_{I_{ij}}(u_h)_tv dxdy&=-\int_{I_i}a(u_h)p_hv_x dxdy+\int_{J_j}\hat{a}_{i+\frac{1}{2},j}\widehat{p_h}_{i+\frac{1}{2},j}v^-_{i+\frac{1}{2},j}dy-\int_{J_j}\hat{a}_{i-\frac{1}{2},j}\widehat{p_h}_{i-\frac{1}{2},j}v^+_{i-\frac{1}{2},j}dy,\\\label{2Dschemeu}
&\:\:\:\:\:-\int_{J_i}b(u_h)p_hv_y dxdy+\int_{I_i}\hat{a}_{i,j+\frac{1}{2}}\widehat{p_h}_{i,j+\frac{1}{2}}v^-_{i,j+\frac{1}{2}}dx-\int_{I_i}\hat{a}_{i,j-\frac{1}{2}}\widehat{p_h}_{i,j-\frac{1}{2}}v^+_{i,j-\frac{1}{2}}dx,\\ \label{2Dschemep}
\int_{P_{i-\frac{1}{2},j}}p_hw dxdy&=-\int_{P_{i-\frac{1}{2},j}}A(u_h)w_x dxdy+ \int_{J_j}A(u_h(\tilde{x}_{i}))w^-_{i}dy-\int_{J_j}A(u_h(\tilde{x}_{i-1}))w^+_{i-1}dy,\\ \label{2Dschemeq}
\int_{Q_{i,j-\frac{1}{2}}}q_hz dxdy&=-\int_{Q_{i,j-\frac{1}{2}}}B(u_h)z_y dxdy+ \int_{I_i}B(u_h(\tilde{y}_{j}))z^-_{j}dx-\int_{I_i}B(u_h(\tilde{y}_{j-1}))z^+_{j-1}dx.
\end{align}

Then the LDG method on overlapping meshes for \eqref{2DLDGequation} is defined as follows:
\begin{itemize}
\item L-mesh:
find $(u_h,p_h,q_h)\in V_h\times P^L_h\times Q^L_h$, such that for any test functions $(v,w,z)\in V_h\times P^L_h\times Q^L_h$ we have \eqref{2Dschemeu} - \eqref{2Dschemeq},
\item C-mesh:
find $(u_h,p_h,q_h)\in V_h\times P^C_h\times Q^C_h$, such that for any test functions $(v,w,z)\in V_h\times P^C_h\times Q^C_h$ we have \eqref{2Dschemeu} - \eqref{2Dschemeq}.
\end{itemize}
We denoted $\displaystyle u^+_{i-\frac{1}{2},j}, u^-_{i+\frac{1}{2},j}, u^+_{i,j-\frac{1}{2}}, u^+_{i,j+\frac{1}{2}} $ as the traces of $u\in V_h$ on the four edges of $I_{ij}$, respectively. Likewise for the traces of along the vertical edges of $P_{i-\frac{1}{2},j}$ and the horizontal edges of $Q_{i,j-\frac{1}{2}}$. Moreover, we use $[u]=u^+-u^-$ and $\displaystyle \{u\}=\frac{1}{2}(u^++u^-)$ as the jump and average of $u$ at the cells interfaces, receptively.

The numerical flux $\hat{a}$ along the edge $x_{i-\frac{1}{2}}$ is defined as
\begin{equation}\label{2daflux}
\hat{a}_{i-\frac{1}{2},j}=\frac{[A(u_h)]_{i-\frac{1}{2},j}}{[u_h]_{i-\frac{1}{2}},j}.
\end{equation}
Similarly, the numerical flux $\hat{b}$ along the edge $y_{j-\frac{1}{2}}$ is defined as
\begin{equation}\label{2dbflux}
\hat{b}_{j-\frac{1}{2}}=\frac{[B(u_h)]_{i,j-\frac{1}{2}}}{[u_h]_{i,j-\frac{1}{2}}}.
\end{equation}
Also, we choose the numerical flux
\begin{equation}\label{2dppenalty}
\widehat{p_h}_{i-\frac{1}{2},j}=p_h(x_{i-\frac{1}{2}},y)+\frac{\alpha_{i-\frac{1}{2},j}}{\Delta\tilde{x}_{i-\frac{1}{2},j}}[u_h]_{i-\frac{1}{2},j},
\end{equation}
and
\begin{equation}\label{2dqpenalty}
\widehat{q_h}_{i,j-\frac{1}{2}}=q_h(x,y_{j-\frac{1}{2}})+\frac{\alpha_{i,j-\frac{1}{2}}}{\Delta\tilde{y}_{i,j-\frac{1}{2},j}}[u_h]_{i,j-\frac{1}{2}},
\end{equation}
where $[s]_{i-\frac{1}{2},j}=s^+_{i-\frac{1}{2},j}-s^-_{i-\frac{1}{2},j}$ stands for the jump of the function $s$ across the cell boundary $\{x_{i-\frac{1}{2}}\}\times J_j$. Similarly for $[s]_{i,j-\frac{1}{2}}$.

To obtain the stability analysis and error estimates, we can follow the same analyses for the problem in one-dimensional space. Therefore, we will omit the proof and only state the results in the following two.

\begin{theorem}
The LDG method introduced \eqref{2Dschemeu}, \eqref{2Dschemep} and \eqref{2Dschemeq} with the boundary conditions \eqref{2DNeumannBC} and \eqref{2DDirichletBC} are stable and
\[\frac{1}{2}\frac{d}{dt}||u_h||^2+||p_h||^2+||q_h||^2\leq 0.\]
\end{theorem}
\begin{theorem}
Suppose the exact solution for linear parabolic equation \eqref{2DLDGequation} with $a(u)=b(u)=1$ satisfies $u\in C^{k+1}(\Omega)$ and the finite element space is made up of piecewise polynomial of degree k in each cell. The numerical solutions satisfy \eqref{2Dschemeu} - \eqref{2Dschemeq}. Then the error between the numerical and exact solutions satisfies
\[||u-u_h||+\int_0^T||p-p_h||+||q-q_h|| dt \leq Ch^k,\]
where $C$ is independent of $h$.
\end{theorem}

\section{Numerical Experiments}\label{sec:Exmaple}
In this section, several numerical experiments will be given to demonstrate the stability and accuracy of the new LDG method on overlapping mesh for non-periodic boundary conditions.
\begin{ex}\label{exCFL} To compare the CFL number for two different meshes, we consider third-order SSP Runge-Kutta time discretization \cite{SSPRK} to solve the following heat equation in one space dimension
\begin{equation}
u_t=u_{xx},\:\:\:\:x\in[0,2\pi],
\end{equation}
subject to Neumann and Dirichlet boundary conditions
\begin{align*}
u_x(0,t)&=u_x(2\pi,t)=0,\\
u(0,t)&=u(2\pi,t)=0,
\end{align*}
respectively.
\end{ex}
We consider a uniform mesh and use linear polynomial $P^1$ with the dual mesh is generated by using the midpoint of the primitive mesh to compare the CFL number to the two different meshes. The results in Table \ref{table:exCFLresults} shows that the CFL number for C-mesh is larger than that for L-mesh. This is mainly because the small cell effect does not works for C-mesh as we combine the small cells near the boundary with its neighbor.
\begin{table}[H]
\begin{center}
\begin{tabular}{|c|c|cc|cc|cc|cc|}
\hline
\multirow{3}{*}{CFL}&\multirow{3}{*}{N}&\multicolumn{4}{c|}{L-mesh}&\multicolumn{4}{c|}{C-mesh } \\ \cline{3-10}
&&\multicolumn{2}{c|}{Neumann}&\multicolumn{2}{c|}{Dirichlet }&\multicolumn{2}{c|}{Neumann}&\multicolumn{2}{c|}{Dirichlet }   \\ \cline{3-10}
 & &$ L^2 $ norm & order& $ L^2 $ norm & order&$ L^2 $ norm & order& $ L^2 $ norm & order\\
\hline
%&10	&9.55E-02	&-		&2.12E-02	&-	 	&4.11E-02	&-		&5.04E-02	&-	\\
%&20	&4.66E-02	&1.03	&4.61E+01	&-		&1.73E-02	&1.25	&2.21E-02	&1.29\\
%0.1&40 &2.30E-02	&1.02	&1.08E+14	&-		&8.34E-03	&1.06	&9.51E-03	&1.22\\
%&80	&1.14E-02	&1.01	&2.63E+70	&-		&4.21E-03	&0.98	&4.50E-03	&1.08\\
%&160&5.67E-03	&1.00	&INF		&-		&2.12E-03	&0.92	&2.18E-03	&1.04\\
%\hline
&10	&1.65E-01	&-		&9.22E+00	&-	 	&4.68E-02	&-		&6.47E-02	&-\\
&20	&1.37E+01	&-		&1.45E+07	&-		&1.83E-02	&1.36	&2.36E-02	&1.46\\
0.25&40 &2.06E+12	&-		&2.27E+33	&-		&8.48E-03	&1.11	&9.70E-03	&1.28\\
&80	&1.17E+59	&-		&1.31E+143	&-		&4.23E-03	&1.00	&4.52E-03	&1.10\\
&160&INF		&-		&Nan		&-		&2.12E-03	&0.99	&2.19E-03	&1.05\\
\hline
\end{tabular}

\caption{\label{table:exCFLresults}Example \ref{exCFL}: The results of CFL testing.}
\end{center}
\end{table}
For all the following numerical experiments, we use piecewise polynomials of degree $k=1,2.$ Moreover, we consider third-order SSP Runge-Kutta time discretization \cite{SSPRK} with $\Delta t = 0.01\Delta x^2$ to reduce the time error.
\begin{ex}\label{exNeumann} We solve the following heat equation in one space dimension
\begin{equation}
\begin{cases}
u_t=u_{xx},\:\:\:\:x\in[0,2\pi],\\
u(x,0)=\cos(x),
\end{cases}
\end{equation}
with Neumann boundary condition
\begin{equation}\label{exNeumannBC}
u_x(0,t)=u_x(2\pi,t)=0
\end{equation}
Clearly, the exact solution is
\[u(x,t)=e^{-t}\cos(x).\]
\end{ex}
We consider a uniform mesh and take $\xi_0=0$ in \eqref{pointondual}, i.e. the dual mesh is generated by using the midpoint of the primitive mesh. We take the final time $T=0.5$ and compute the $L^2$-norm of the error between the numerical and exact solutions. In Table \ref{table:Neumannuniformmidpoint}, the results show that we can only obtain suboptimal accuracy if $k$ is an odd number with the penalty parameter $\alpha=0$. One way to obtain the optimal accuracy is to choose $\alpha\neq 0$.

\begin{table}[H]
\begin{center}
\begin{tabular}{|c|c|cc|cc|cc|cc|}
\hline
\multirow{3}{*}k&\multirow{3}{*}{N}&\multicolumn{4}{c|}{L-mesh}&\multicolumn{4}{c|}{C-mesh } \\ \cline{3-10}
&&\multicolumn{2}{c|}{no penalty}&\multicolumn{2}{c|}{$\alpha=1.0$ }&\multicolumn{2}{c|}{no penalty}&\multicolumn{2}{c|}{$\alpha=1.0$ }   \\ \cline{3-10}
 & &$ L^2 $ norm & order& $ L^2 $ norm & order&$ L^2 $ norm & order& $ L^2 $ norm & order\\
\hline
&10	&9.51E-02	&-		&2.12E-02	&-	 	&3.91E-02	&-		&2.78E-02	&-	\\
&20	&4.66E-02	&1.03	&4.61E-03	&2.20	&1.8E-02	&1.08	&6.98E-03	&1.99\\
1&40&2.30E-02	&1.02	&1.08E-03	&2.09	&8.66E-03	&1.10	&1.64E-03	&2.10\\
&80	&1.14E-02	&1.01	&2.63E-04	&2.04	&4.24E-03	&1.03	&3.91E-04	&2.06\\
&160&5.67E-03	&1.00	&6.49E-05	&2.02	&2.13E-03	&0.99	&9.52E-05	&2.04\\
\hline
&10	&1.29E-03	&-		&9.37E-04	&-	 	&2.26E-03	&-		&1.87E-03	&-\\
&20	&1.60E-04	&3.01	&1.14E-04	&3.04	&3.56E-04	&2.67	&1.76E-04	&3.41\\
2&40&1.99E-05	&3.00	&1.41E-05	&3.01	&5.63E-05	&2.66	&2.00E-05	&3.14\\
&80	&2.49E-06	&3.00	&1.76E-06	&3.00	&9.26E-06	&2.60	&2.46E-06	&3.03\\
&160&3.12E-07	&3.00	&2.20E-07	&3.00	&1.57E-06	&2.56	&3.06E-07	&3.00\\
\hline
\end{tabular}

\caption{\label{table:Neumannuniformmidpoint}Example \ref{exNeumann}: midpoint with Neumann boundary condition.}
\end{center}
\end{table}

In addition, we also take $\alpha=0$ and $\xi_0=0.1$ which is closed to 0 and $\xi_0=\sqrt{3}/3$ which is away from 0. The results in Table  \ref{table:Neumannuniformnotmidpoint} demonstrated that we may recover the optimal convergence rates if we take $\xi_0\neq 0$. However, we may observe only the $\displaystyle k+\frac{1}{2}$ order of accuracy when we consider the C-mesh treatment without the penalty term.  % Moreover, we can observe that the errors for $\xi_0=\sqrt{3}/3$  are slightly less than those for $\xi_0=0.05$ and $\xi_0=0.25$.

\begin{table}[H]
\begin{center}
\begin{tabular}{|c|c|cc|cc|cc|cc|}
\hline
\multirow{3}{*}k&\multirow{3}{*}{N}&\multicolumn{4}{c|}{L-mesh}&\multicolumn{4}{c|}{C-mesh} \\ \cline{3-10}
&&\multicolumn{2}{c|}{$\xi_0=0.1$}&\multicolumn{2}{c|}{$\xi_0=\sqrt{3}/3$ }&\multicolumn{2}{c|}{$\xi_0=0.1$}&\multicolumn{2}{c|}{$\xi_0=\sqrt{3}/3$ }   \\ \cline{3-10}
 & &$ L^2 $ norm & order& $ L^2 $ norm & order&$ L^2 $ norm & order& $ L^2 $ norm & order\\
\hline
&10	&4.18E-02	&-	 	&1.87E-02	&-	 	&4.09E-02	&-		&3.08E-02	&-	\\
&20	&9.24E-03	&2.18	&4.05E-03	&2.09	&1.71E-02	&1.26	&6.99E-03	&2.14\\
1&40&2.25E-03	&2.04	&1.05E-03	&2.07	&7.31E-03	&1.22	&1.62E-03	&2.11\\
&80	&5.65E-04	&1.99	&2.55E-04	&2.04	&2.21E-03	&1.72	&3.87E-04	&2.06\\
&160&1.42E-04	&1.99	&6.28E-05	&2.02	&5.66E-04	&1.98	&9.45E-05	&2.03\\
\hline
&10	&1.51E-03	&-		&1.29E-03	&-	 	&2.29E-03	&-		&2.76E-03	&-\\
&20	&1.79E-04	&3.07	&1.55E-04	&3.05	&6.59E-04	&2.68	&4.60E-04	&2.59\\
2&40&2.22E-05	&3.02	&1.93E-05	&3.01	&5.65E-05	&2.67	&7.73E-05	&2.57\\
&80	&2.76E-06	&3.00	&2.41E-06	&3.00	&9.26E-06	&2.61	&1.33E-06	&2.54\\
&160&3.45E-07	&3.00	&3.01E-07	&3.00	&1.57E-06	&2.56	&2.30E-06	&2.52\\
\hline
\end{tabular}

\caption{\label{table:Neumannuniformnotmidpoint}Example \ref{exNeumann}: $\xi_0=0.1$ and $\xi_0=\sqrt{3}/3$ with Neumann boundary condition.}
\end{center}
\end{table}

\begin{ex}\label{exDirichlet} We solve
\begin{equation}
\begin{cases}
u_t=u_{xx},\:\:\:\:x\in[0,2\pi],\\
u(x,0)=\sin(x),
\end{cases}
\end{equation}subject to Dirichlet boundary condition
\begin{equation}\label{exDirichletBC}
u(0,t)=u(2\pi,t)=0.
\end{equation}
Clearly, the exact solution is
\[u(x,t)=e^{-t}\sin(x).\]
\end{ex}
We also consider a uniform mesh and take $\xi_0=0$ in \eqref{pointondual}. Then, we choose $\alpha=0$, and take $\xi_0=0.1$ which is closed to 0 and $\xi_0=\sqrt{3}/3$ which is away from 0. According to Table \ref{table:Dirichletuniformmidpoint} and Table \ref{table:Dirichletuniformnotmidpoint}, we can observe the same outcomes discussed for the Neumann boundary condition in Example \ref{exNeumann}. However, we can obtain the optimal accuracy when we consider this case without the penalty term.

\begin{table}[H]
\begin{center}
\begin{tabular}{|c|c|cc|cc|cc|cc|}
\hline
\multirow{3}{*}k&\multirow{3}{*}{N}&\multicolumn{4}{c|}{L-mesh}&\multicolumn{4}{c|}{C-mesh } \\ \cline{3-10}
&&\multicolumn{2}{c|}{no penalty}&\multicolumn{2}{c|}{$\alpha=1.0$ }&\multicolumn{2}{c|}{no penalty}&\multicolumn{2}{c|}{$\alpha=1.0$ }   \\ \cline{3-10}
 & &$ L^2 $ norm & order& $ L^2 $ norm & order&$ L^2 $ norm & order& $ L^2 $ norm & order\\
\hline
&10	&7.19E-02	&-		&1.82E-02	&-	 	&5.04E-02	&-		&3.16E-02	&-\\
&20	&3.54E-02	&1.02	&4.26E-03	&2.10	&2.21E-02	&1.29	&6.45E-03	&2.30\\
1&40&1.76E-02	&1.00	&1.04E-03	&2.04	&9.51E-03	&1.22	&1.50E-03	&2.10\\
&80	&8.81E-03	&1.00	&2.57E-04	&2.01	&4.50E-03	&1.08	&3.70E-04	&2.02\\
&160&4.40E-03	&1.00	&6.42E-05	&2.00	&2.18E-03	&1.04	&9.23E-05	&2.00\\
\hline
&10	&1.32E-03	&-		&9.75E-04	&-	 	&1.96E-03	&-		&1.59E-03	&-\\
&20	&1.63E-04	&3.01	&1.16E-04	&3.08	&2.41E-04	&3.03	&2.03E-04	&2.97\\
2&40&2.02E-05	&3.01	&1.42E-05	&3.03	&2.99E-05	&3.01	&2.30E-05	&3.14\\
&80	&2.51E-06	&3.01	&1.76E-06	&3.01	&3.73E-06	&3.00	&2.68E-06	&3.10\\
&160&3.13E-07	&3.00	&2.20E-07	&3.00	&4.66E-07	&3.00	&3.20E-07	&3.06\\
\hline
\end{tabular}

\caption{\label{table:Dirichletuniformmidpoint}Example \ref{exDirichlet} midpoint with Dirichlet boundary condition.}
\end{center}
\end{table}

\begin{table}
\begin{center}
\begin{tabular}{|c|c|cc|cc|cc|cc|}
\hline
\multirow{3}{*}k&\multirow{3}{*}{N}&\multicolumn{4}{c|}{L-mesh}&\multicolumn{4}{c|}{C-mesh } \\ \cline{3-10}
&&\multicolumn{2}{c|}{$\xi_0=0.1$}&\multicolumn{2}{c|}{$\xi_0=\sqrt{3}/3$ }&\multicolumn{2}{c|}{$\xi_0=0.1$}&\multicolumn{2}{c|}{$\xi_0=\sqrt{3}/3$ }   \\ \cline{3-10}
 & &$ L^2 $ norm & order& $ L^2 $ norm & order&$ L^2 $ norm & order& $ L^2 $ norm & order\\
\hline
&10	&3.85E-02	&-	 &1.58E-02	&-	 	&5.34E-02	&-		&3.91E-02	&-\\
&20	&9.56E-03	&2.01&3.95E-03	&2.00	&2.21E-02	&1.27	&1.26E-02	&1.63\\
1&40&2.33E-03	&2.04&9.88E-04	&2.00	&8.97E-03	&1.30	&4.02E-03	&1.65\\
&80	&5.77E-04	&2.01&2.47E-04	&2.00	&2.73E-03	&1.71	&1.37E-03	&1.55\\
&160&1.44E-04	&2.00&6.18E-05	&2.00	&7.52E-04	&1.86	&4.74E-04	&1.53\\
\hline
&10	&1.61E-03	&-	 &1.36E-03	&-	 	&1.99E-03	&-		&2.30E-03	&-\\
&20	&1.87E-04	&3.11&1.61E-04	&3.08	&2.46E-04	&3.02	&2.61E-04	&3.14\\
2&40&2.27E-05	&3.04&1.96E-05	&3.03	&3.05E-05	&3.01	&3.07E-05	&3.09\\
&80	&2.80E-06	&3.02&2.43E-06	&3.01	&3.81E-06	&3.00	&3.71E-06	&3.04\\
&160&3.47E-07	&3.01&3.02E-07	&3.01	&4.76E-07	&3.00	&4.57E-07	&3.02\\
\hline
\end{tabular}

\caption{\label{table:Dirichletuniformnotmidpoint}Example \ref{exDirichlet} $\xi_0=0.1$ , and $\xi_0=\sqrt{3}/3$ with Dirichlet boundary condition.}
\end{center}
\end{table}

\begin{ex}\label{ex2DNeumann} We solve the following heat equation in two space dimension
\begin{equation}
\begin{cases}
u_t=u_{xx}+u_{yy},\:\:\:\:(x,y)\in[0,2\pi]\times[0,2\pi],\\
u(x,y,0)=\cos(x)\cos(y),
\end{cases}
\end{equation}
subject to Neumann boundary condition
\begin{equation}\label{ex2DNeumannBC}
u_x(0,y,t)=u_x(2\pi,y,t)=0\:\:\text{and}\:\:
u_y(x,0,t)=u_y(x,2\pi,t)=0.
\end{equation}
Clearly, the exact solution is
\[u(x,y,t)=e^{-2t}\cos(x)\cos(y).\]
\end{ex}
We also consider a uniform mesh, and take $\xi_0=0$ and $\eta_0=0$ to examine the dual meshes which were generated by the midpoint in both directions. We choose the final time to be $T=0.1$. Table \ref{table:2DNeumannuniformmidpoint} shows suboptimal accuracy if $k$ is an odd number with the penalty parameter $\alpha=0$. Also, one way to obtain the optimal accuracy is choosing $\alpha\neq 0$. However, when the C-mesh was used we can only obtain the $\displaystyle k+\frac{1}{2}$ order of accuracy.

\begin{table}[H]
\begin{center}
\begin{tabular}{|c|c|cc|cc|cc|cc|}
\hline
\multirow{3}{*}k&\multirow{3}{*}{N}&\multicolumn{4}{c|}{L-mesh}&\multicolumn{4}{c|}{C-mesh } \\ \cline{3-10}
&&\multicolumn{2}{c|}{no penalty}&\multicolumn{2}{c|}{$\alpha=1.0$ }&\multicolumn{2}{c|}{no penalty}&\multicolumn{2}{c|}{$\alpha=1.0$ }   \\ \cline{3-10}
 & &$ L^2 $ norm & order& $ L^2 $ norm & order&$ L^2 $ norm & order& $ L^2 $ norm & order\\
\hline
&16	&3.72E-01	&-		&1.34E-01	&-		&5.51E-01	&-		&5.52E-01	&-\\
1&64&1.36E-01	&1.45	&3.98E-02	&1.78	&1.92E-01	&1.52	&1.96E-01	&1.50\\
&256&5.51E-02	&1.30	&1.12E-02	&1.83	&6.99E-02	&1.45	&6.14E-02	&1.67\\
&1024&2.54E-02	&1.11	&2.81E-03	&2.00	&2.17E-02	&1.69	&2.00E-02	&1.62\\
\hline
&16	&2.06E-01	&-		&2.26E-01	&-		&7.85E-01	&-		&6.23E-01	&-\\
&64	&3.40E-02	&2.59	&4.51E-02	&2.34	&1.44E-01	&2.39	&1.12E-01	&2.47\\
2&256&5.20E-03	&2.71	&7.91E-03	&2.51	&2.59E-02	&2.48	&2.02E-02	&2.48\\
&1024&7.51E-04	&2.79	&1.14E-03	&2.79	&4.63E-03	&2.48	&3.65E-03	&2.47\\
&4096&9.52E-05	&2.97	&1.58E-04	&2.85	&8.21E-04	&2.49	&6.54E-04	&2.48\\
&16384&1.19E-05	&2.99	&1.98E-05	&2.99	&1.47E-04	&2.49	&1.18E-04	&2.48\\
\hline
\end{tabular}

\caption{\label{table:2DNeumannuniformmidpoint}Example \ref{ex2DNeumann}: midpoint with Neumann boundary condition.}
\end{center}
\end{table}

Moreover, we take $\alpha=0$, and consider the combination of $\xi_0=0, 0.5$  and  $\eta_0=0, 0.5$ The results in Table  \ref{table:2DNeumannuniformnotmidpoint} demonstrated that we may recover optimal convergence rates we if take $\xi_0\neq 0$ and $\eta_0\neq 0$. However, when we use the C-mesh, we can only obtain $\displaystyle k+\frac{1}{2}$ order of accuracy as shown in Table  \ref{table:2DNeumannuniformnotmidpointcombine}.

\begin{table}[H]
\begin{center}
\begin{tabular}{|c|c|cc|cc|cc|}
\hline
\multirow{3}{*}k&\multirow{3}{*}{N}&\multicolumn{6}{c|}{L-mesh}\\ \cline{3-8}
&&\multicolumn{2}{c|}{$\xi_0 = 0$, $\eta_0=0.5$ }&\multicolumn{2}{c|}{$\xi_0 = 0.5$, $\eta_0=0$  } &\multicolumn{2}{c|}{$\xi_0 = 0.5$, $\eta_0=0.5$ } \\ \cline{3-8}
 & &$ L^2 $ norm & order& $ L^2 $ norm & order& $ L^2 $ norm & order\\
\hline
&16		&3.48E-01	&-		&3.47E-01	&-		&3.22E-01	&-	\\
1&64	&1.15E-01	&1.59	&1.16E-01	&1.57  	&9.04E-02	&1.83	\\
&256	&4.22E-02	&1.44	&4.27E-02	&1.45	&2.28E-02	&1.99	\\
&1024	&1.84E-02	&1.19	&1.85E-02	&1.20	&5.55E-03	&2.04	\\
\hline
&16		&1.76E-01	&-		&1.78E-01	&-		&1.69E-01	&-	\\
&64		&3.76E-02	&2.66	&2.78E-02	&2.68	&2.92E-02	&2.53	\\
2&256	&3.74E-03	&2.88	&4.06E-03	&2.78	&4.63E-03	&2.66	\\
&1024	&5.03E-04	&2.90	&5.78E-04	&2.81	&5.93E-04	&2.96	\\
&4096	&6.35E-05	&2.98	&7.35E-05	&2.98	&7.59E-05	&2.97	\\
&16384	&7.99E-06	&2.99	&9.29E-06	&2.99	&9.52E-06	&2.99	\\
\hline
\end{tabular}

\caption{\label{table:2DNeumannuniformnotmidpoint}Example \ref{ex2DNeumann}: combination of $\xi_0=0, 0.5$  and  $\eta_0=0, 0.5$ with Neumann boundary condition on L-mesh.}
\end{center}
\end{table}

\begin{table}[H]
\begin{center}
\begin{tabular}{|c|c|cc|cc|cc|}
\hline
\multirow{3}{*}k&\multirow{3}{*}{N}&\multicolumn{6}{c|}{C-mesh}\\ \cline{3-8}
&&\multicolumn{2}{c|}{$\xi_0 = 0$, $\eta_0=0.5$ }&\multicolumn{2}{c|}{$\xi_0 = 0.5$, $\eta_0=0$  } &\multicolumn{2}{c|}{$\xi_0 = 0.5$, $\eta_0=0.5$ } \\ \cline{3-8}
 & &$ L^2 $ norm & order& $ L^2 $ norm & order& $ L^2 $ norm & order\\
\hline
&16		&5.54E-01	&-		&5.56E-01	&-		&5.57E-01	&-	\\
1&64	&1.94E-01	&1.51	&1.93E-01	&1.53  	&1.84E-02	&1.60	\\
&256	&6.13E-02	&1.66	&5.70E-02	&1.76	&5.97E-02	&1.62	\\
&1024	&1.99E-02	&1.62	&1.83E-02	&1.64	&1.89E-03	&1.66	\\
\hline
&16		&5.48E-01	&-		&4.85E-01	&-		&3.53E-01	&-	\\
&64		&8.88E-02	&2.62	&7.85E-02	&2.63	&5.42E-02	&2.70	\\
2&256	&1.49E-02	&2.57	&1.25E-02	&2.65	&8.57E-02	&2.66	\\
&1024	&2.59E-03	&2.53	&2.06E-03	&2.60	&1.46E-03	&2.56	\\
&4096	&4.48E-04	&2.53	&3.45E-04	&2.58	&2.58E-04	&2.50	\\
&16384	&8.02E-05	&2.48	&6.10E-05	&2.50	&4.57E-05	&2.50	\\
\hline
\end{tabular}

\caption{\label{table:2DNeumannuniformnotmidpointcombine}Example \ref{ex2DNeumann}: combination of $\xi_0=0, 0.5$  and  $\eta_0=0, 0.5$ with Neumann boundary condition on C-mesh.}
\end{center}
\end{table}

\begin{ex}\label{ex2DDirichlet} We solve
\begin{equation}
\begin{cases}
u_t=u_{xx}+u_{yy},\:\:\:\:(x,y)\in[0,2\pi]\times[0,2\pi],\\
u(x,y,0)=\sin(x)\sin(y),
\end{cases}
\end{equation}
subject to Dirichlet boundary condition
\begin{equation}\label{ex2DDirichletBC}
u(0,y,t)=u(2\pi,y,t)=0\:\:\text{and}\:\:
u(x,0,t)=u(x,2\pi,t)=0
\end{equation}
Clearly, the exact solution is
\[u(x,y,t)=e^{-2t}\sin(x)\sin(y).\]
\end{ex}
We also explore the problem on a uniform mesh and take $\xi_0=0$ and $\eta_0=0$ to examine the dual meshes that were generated by the midpoint in both directions. We take $\alpha=0$ and consider the combination of $\xi_0=0, 0.5$  and  $\eta_0=0, 0.5$. As stated in Table \ref{table:2DDirichletuniformmidpoint}, Table \ref{table:2DDirichletuniformnotmidpoint} and Table \ref{table:2DDirichletuniformnotmidpointcombine}, we can observe the same results discussed for problems with Neumann boundary condition in Example \ref{ex2DNeumann}.

\begin{table}
\begin{center}
\begin{tabular}{|c|c|cc|cc|cc|cc|}
\hline
\multirow{3}{*}k&\multirow{3}{*}{N}&\multicolumn{4}{c|}{L-mesh}&\multicolumn{4}{c|}{C-mesh } \\ \cline{3-10}
&&\multicolumn{2}{c|}{no penalty}&\multicolumn{2}{c|}{$\alpha=1.0$ }&\multicolumn{2}{c|}{no penalty}&\multicolumn{2}{c|}{$\alpha=1.0$ }   \\ \cline{3-10}
 & &$ L^2 $ norm & order& $ L^2 $ norm & order&$ L^2 $ norm & order& $ L^2 $ norm & order\\
\hline
&16&3.55E-01	&-		&8.83E-02	&-		&6.62E-01	&-		&6.63E-01	&-\\
1&64&1.08E-01	&1.73	&2.18E-02	&2.01	&18.2E-01	&1.86	&1.84E-01	&1.85\\
&256&4.33E-02	&1.30	&5.38E-03	&2.01	&4.93E-02	&1.86	&5.20E-02	&1.82\\
&1024&2.05E-02	&1.08	&1.33E-03	&2.01	&1.57E-02	&1.64	&1.45E-02	&1.85\\
\hline
&16&5.12E-02	&-		&1.69E-01	&-		&6.17E-01	&-		&5.29E-01	&-\\	
&64&9.81E-03	&2.39	&3.24E-02	&2.39	&1.27E-01	&2.28	&1.10E-01	&2.27\\
2&256&1.58E-03	&2.63	&5.37E-03	&2.59	&2.22E-02	&2.51	&1.97E-02	&2.48\\
&1024&2.25E-04	&2.81	&8.54E-04	&2.65	&3.82E-03	&2.54	&3.45E-03	&2.51\\
&4096&3.00E-05	&2.90	&1.07E-04	&2.99	&6.65E-04	&2.52	&6.07E-04	&2.51\\
&16384&3.75E-06	&3.00	&1.34E-05	&2.99	&1.17E-04	&2.51	&1.07E-04	&2.50\\
\hline
\end{tabular}

\caption{\label{table:2DDirichletuniformmidpoint}Example \ref{ex2DDirichlet}: midpoint with Dirichlet boundary condition.}
\end{center}
\end{table}

\begin{table}[H]
\begin{center}
\begin{tabular}{|c|c|cc|cc|cc|}
\hline
\multirow{3}{*}k&\multirow{3}{*}{N}&\multicolumn{6}{c|}{L-mesh}\\ \cline{3-8}
&&\multicolumn{2}{c|}{$\xi_0 = 0$, $\eta_0=0.5$ }&\multicolumn{2}{c|}{$\xi_0 = 0.5$, $\eta_0=0$  } &\multicolumn{2}{c|}{$\xi_0 = 0.5$, $\eta_0=0.5$ } \\ \cline{3-8}
 & &$ L^2 $ norm & order& $ L^2 $ norm & order& $ L^2 $ norm & order\\
\hline
&16		&3.51E-01	&-		&3.54E-01	&-		&3.47E-01	&-	\\
1&64	&9.62E-01	&1.86	&1.07E-01	&1.72  	&8.04E-02	&2.05	\\
&256	&3.40E-02	&1.50	&4.27E-02	&1.33	&2.09E-02	&2.01	\\
&1024	&1.49E-02	&1.18	&1.69E-02	&1.33	&5.23E-03	&2.00	\\
\hline
&16		&5.20E-02	&-		&5.20E-02	&-		&5.30E-02	&-	\\
&64		&9.66E-03	&2.42	&9.66E-03	&2.43	&9.54E-03	&2.47	\\
2&256	&1.54E-03	&2.64	&1.54E-03	&2.65	&1.50E-03	&2.67	\\
&1024	&2.12E-04	&2.85	&2.12E-04	&2.86	&1.99E-04	&2.91	\\
&4096	&2.79E-05	&2.93	&2.79E-05	&2.93	&2.56E-05	&2.96\\
&16384	&3.50E-06	&2.99	&3.49E-06	&3.00	&3.20E-06	&3.00	\\
\hline
\end{tabular}

\caption{\label{table:2DDirichletuniformnotmidpoint}Example \ref{ex2DDirichlet}: combination of $\xi_0=0, 0.5$  and  $\eta_0=0, 0.5$  with Dirichlet boundary condition on L-mesh.}
\end{center}
\end{table}
\begin{table}[H]
\begin{center}
\begin{tabular}{|c|c|cc|cc|cc|}
\hline
\multirow{3}{*}k&\multirow{3}{*}{N}&\multicolumn{6}{c|}{C-mesh}\\ \cline{3-8}
&&\multicolumn{2}{c|}{$\xi_0 = 0$, $\eta_0=0.5$ }&\multicolumn{2}{c|}{$\xi_0 = 0.5$, $\eta_0=0$  } &\multicolumn{2}{c|}{$\xi_0 = 0.5$, $\eta_0=0.5$ } \\ \cline{3-8}
 & &$ L^2 $ norm & order& $ L^2 $ norm & order& $ L^2 $ norm & order\\
\hline
&16		&6.62E-01	&-		&6.62E-01	&-		&6.62E-01	&-	\\
1&64	&1.81E-01	&1.86	&1.81E-01	&1.87  	&1.80E-02	&1.88	\\
&256	&4.91E-02	&1.88	&4.81E-02	&1.91	&4.77E-02	&1.92	\\
&1024	&1.59E-02	&1.62	&1.32E-02	&1.87	&1.28E-03	&1.91	\\
\hline
&16		&6.48E-01	&-		&5.81E-01	&-		&6.14E-01	&-	\\
&64		&1.31E-01	&2.30	&1.12E-01	&2.37	&1.17E-01	&2.39	\\
2&256	&2.34E-02	&2.50	&1.85E-02	&2.60	&1.96E-02	&2.58	\\
&1024	&4.07E-03	&2.53	&3.03E-03	&2.60	&3.24E-03	&2.60	\\
&4096	&7.10E-04	&2.52	&5.12E-04	&2.57	&5.48E-04	&2.56	\\
&16384	&1.25E-05	&2.51	&8.83E-05	&2.54	&9.45E-05	&2.54	\\
\hline
\end{tabular}

\caption{\label{table:2DDirichletuniformnotmidpointcombine}Example \ref{ex2DDirichlet}: combination of $\xi_0=0, 0.5$  and  $\eta_0=0, 0.5$  with Dirichlet boundary condition on C-mesh.}
\end{center}
\end{table}

\section{Conclusion}\label{sec:conclusion}
In this paper, we demonstrate the algorithms of the LDG method on overlapping mesh for Neumann and Dirichlet boundary conditions. The scheme for each boundary condition is stable with adjusted boundary cells but the order of accuracy may not be optimal. We observed that C-mesh could yield larger time step than L-mesh. However, in some cases, C-mesh cannot yield optimal convergence rates. With a positive penalty parameter, both meshes can result in optimal convergence rates.\\

\textbf{Acknowledgements}  This work is supported by NSF grant DMS-1818467 KMITL Research Fund, Research Seed Grant for New Lecturer.

\end{document}